\documentclass[12pt, reqno]{amsart}
\usepackage{amsmath, amsthm, amscd, amsfonts, amssymb, graphicx, color}
\usepackage[bookmarksnumbered, colorlinks, plainpages]{hyperref}

\input{mathrsfs.sty}

\newtheorem{theorem}{Theorem}[section]
\newtheorem{lemma}[theorem]{Lemma}

\newtheorem{corollary}[theorem]{Corollary}
\theoremstyle{definition}

\newtheorem{example}[theorem]{Example}

\theoremstyle{remark}

\numberwithin{equation}{section}
\begin{document}

\title [Some extensions of the operator entropy type inequalities]{Some extensions of the operator entropy type inequalities}

\author[M. Bakherad, A. Morassaei]{Mojtaba Bakherad$^1$ and Ali Morassaei$^2$}

\address{$^1$Department of Mathematics, Faculty of Mathematics, University of Sistan and Baluchestan, Zahedan, Iran.}
\email{mojtaba.bakherad@yahoo.com; bakherad@member.ams.org}
\address{$^2$Department of Mathematics, Faculty of Sciences, University of Zanjan, University Blvd., Zanjan 45371-38791, Iran.}
\email{morassaei@znu.ac.ir}
\subjclass[2010]{Primary 47A63, Secondary 46L05, 47A60.}

\keywords{operator entropy; operator monotone function; Bochner integration; entropy inequality; positive linear map.}
\begin{abstract}
In this paper, we define the generalized relative operator entropy and investigate some of its properties such as subadditivity and homogeneity. As application of our result, we obtain the information inequality.

In continuation, we  establish some reverses of the operator entropy inequalities under certain conditions by using the Mond--Pe\v{c}ari\'c method.
\end{abstract}

\maketitle



\section{Introduction and preliminaries}
Let ${\mathbb B}({\mathscr H})$ denote the $C^*$-algebra of all bounded linear operators on a complex Hilbert space ${\mathscr H}$ with the identity $I_{\mathscr H}$. In the case when ${\rm dim}{\mathscr{H}}=n$, we identify $\mathbb{B}(\mathscr{H})$ with the full matrix algebra $\mathcal{M}_n(\mathbb{C})$ of all $n\times n$ matrices with entries in the complex field. An operator $A\in{\mathbb B}({\mathscr H})$ is called \textit{positive} if $\langle Ax,x\rangle\geq0$ for all $x\in{\mathscr H }$  and in this case we write $A\geq0$. We write $A>0$ if $A$ is a positive invertible operator.  For self-adjoint operators $A, B\in{\mathbb B}({\mathscr H})$, we say $A\leq B$ if $B-A\geq0$. The Gelfand map $f(t)\mapsto f(A)$ is an isometrical $*$-isomorphism between the $C^*$-algebra
$\mathcal{C}({\textsf{sp}}(A))$ of continuous functions on the spectrum ${\textsf{sp}}(A)$ of a self-adjoint operator $A$ and the $C^*$-algebra generated by $A$ and  $I_{\mathscr H}$. If $f, g\in \mathcal{C}({\textsf{sp}}(A))$, then $f(t)\geq g(t)\,\,(t\in{\textsf{sp}}(A))$ implies that $f(A)\geq g(A)$.\\
Let $f$ be a continuous real valued function defined on an interval $J$. It is called \textit{operator monotone} if $A\leq B$ implies $f(A)\leq f(B)$ for all self-adjoint operators $A, B\in {\mathbb B}({\mathscr H})$ with spectra in $J$; see \cite{banach}  and references therein for some recent results. It said to be \textit{operator concave} if $\lambda f(A)+(1-\lambda)f(B)\leq f(\lambda A+(1-\lambda)B)$ for all self-adjoint operators $A, B\in {\mathbb B}({\mathscr H})$ with spectra in $J$ and all $\lambda\in [0,1]$. Every nonnegative continuous function $f$ is operator monotone on $[0,+\infty)$ if and only if $f$ is operator concave on $[0,+\infty)$; see \cite[Theorem 8.1]{FMPS}.

A linear map $\Phi:\mathbb{B}(\mathscr{H})\to\mathbb{B}(\mathscr{K})$, where $\mathscr{H}$ and $\mathscr{K}$ are complex Hilbert spaces, is called \textit{positive} if $\Phi(A)\geq 0$ whenever $A\geq 0$ and is said to be \textit{normalized} if $\Phi(I_\mathscr{H})=I_\mathscr{K}$. We denote by $\mathbf{P}_N[\mathbb{B}(\mathscr{H}),\mathbb{B}(\mathscr{K})]$ the set of all normalized positive linear maps $\Phi:\mathbb{B}(\mathscr{H})\to\mathbb{B}(\mathscr{K})$. If $\Phi \in \mathbf{P}_N[\mathbb{B}(\mathscr{H}),\mathbb{B}(\mathscr{K})]$ and $f$ is an operator concave function on an interval $J$, then
\begin{equation}\label{jen}
f(\Phi(A))\ge\Phi(f(A))\,\quad\quad(\mbox{Davis-Choi-Jensen's inequality})
\end{equation}
for every selfadjoint operator $A$ on $\mathscr{H}$, whose spectrum is contained in $J$, see also \cite{FMPS, MPP, IMP}.

Let ${\mathscr A}$ be a $C^*$-algebra of operators acting on  a Hilbert space, let $T$ be a locally compact Hausdorff space and  $\mu(t)$ be a Radon measure on $T$. A field $(A_t)_{t\in T}$ of operators in ${\mathscr A}$ is called a \textit{continuous field of operators} if the function $t\mapsto A_t$ is norm continuous on $T$ and the function $t \mapsto\|A_t\|$ is integrable, one can form the Bochner integral $\int_{T}A_t{\rm d}\mu(t)$, which is the unique element in ${\mathscr A}$ such that
\begin{equation}\label{e1}
\varphi\left(\int_TA_t{\rm d}\mu(t)\right)=\int_T\varphi(A_t){\rm d}\mu(t)
\end{equation}
for every linear functional $\varphi$ in the norm dual ${\mathscr A}^*$ of ${\mathscr A}$; see \cite{han}.

In 1850 Clausius \cite{Cl} introduced the notion of entropy in the thermodynamics. Since then several extensions and reformulations have been developed in various disciplines; cf. \cite{ME, LR, L, NU}. There have been investigated the so-called entropy inequalities by some mathematicians; see \cite{BLP, BS, FUR2} and references therein. A relative operator entropy of strictly positive operators $A,B$ was introduced in the
noncommutative information theory by Fujii and Kamei \cite{FK} by
$$
S(A|B) = A^{\frac{1}{2}} \log(A^{-\frac{1}{2}}BA^{-\frac{1}{2}})A^{\frac{1}{2}}.
$$
In the same paper, it is shown that $S(A|B) \le 0$ if $A \ge B$.

Next, recall that $X\natural_q Y$ is defined by $X^{\frac{1}{2}}\left(X^{-\frac{1}{2}}YX^{-\frac{1}{2}}\right)^qX^{\frac{1}{2}}$ for any real number $q$ and any strictly positive operators $X$ and $Y$. For $p\in[0,1]$, the operator $X \natural_p Y$ coincides with the well-known geometric mean of $X, Y$.

Furuta \cite{F} defined the operator Shannon entropy by
$$
S_p(A|B)=A^{\frac{1}{2}}\left(A^{-\frac{1}{2}}BA^{-\frac{1}{2}}\right)^p\log\left(A^{-\frac{1}{2}}BA^{-\frac{1}{2}}\right)A^{\frac{1}{2}}\,,
$$
where $p\in[0, 1]$ and $A, B$ are strictly positive operators on a Hilbert space $\mathscr {H}$. Suppose that   $\mathbf{A}=(A_t)_{t\in T},\mathbf{B}=(B_t)_{t\in T}$ are (continuous) fields of strictly positive operators, $q\in\mathbb R$ and $f$ is a nonnegative operator monotone function on $(0,\infty)$. Then we have the definition of the generalized relative operator entropy
\begin{equation}\label{ShEn321}
\widetilde{S}_q^f(\mathbf{A}|\mathbf{B}):=\int_TS_q^f(A_s|B_s)d\mu(s)\,,
\end{equation}
where $S^f_q(A_s|B_s)=A_s^{1/2}\left(A_s^{-1/2}B_sA_s^{-1/2}\right)^qf\left(A_s^{-1/2}B_sA_s^{-1/2}\right)A_s^{1/2}$. In the discrete case $T=\{1,2,\cdots,n\}$, we get
\begin{equation}\label{ShEn}
S_q^f(\mathbf{A}|\mathbf{B}):=\sum_{j=1}^nS_q^f(A_j|B_j).
\end{equation}
For $q=0$, $f(t)=\log t$ and $A, B>0$, we get the relative operator entropy $S_0^f(A|B)=A^{\frac{1}{2}}\log\left(A^{-\frac{1}{2}}BA^{-\frac{1}{2}}\right)A^{\frac{1}{2}}=S(A|B)$.

Moslehian et. al \cite{MMM} showed the following operator entropy
\begin{align}\label{ShEnIn1}
f&\left[\sum_{j=1}^n(A_j\natural_{p+1}B_j)+t_0\left(I_{\mathscr H}-\sum_{j=1}^nA_j\natural_pB_j\right)\right]-f(t_0)\left(I_{\mathscr H}-\sum_{j=1}^nA_j\natural_pB_j\right)\nonumber\\
&\ge S_p^f(\mathbf{A}|\mathbf{B})\qquad(p\in[0,1]),
\end{align}
where $\mathbf{A}=(A_1,\cdots,A_n)$ and $\mathbf{B}=(B_1,\cdots,B_n)$ are finite sequences of strictly positive operators such that $\sum_{j=1}^nA_j=\sum_{j=1}^nB_j=I_{\mathscr H}$,  $f$ is a nonnegative operator monotone function on $(0,\infty)$ and $t_0$ is a positive fixed real number.\\

We present some extensions of the operator entropy inequality. Also, we show some reverses of the operator entropy inequalities under certain conditions by using the Mond--Pe\v{c}ari\'c method. In this direction, we show a reverse of \eqref{ShEnIn1}.

\section{Some extensions of the operator entropy inequality}
First, we present variational form of $S_q^f(A|B)$ where $A$ and $B$ are two strictly positive operator in $\mathbb{B}(\mathscr{H})$ and $q$ is an arbitrary real number.
\begin{lemma}\label{l1}
If $A$ and $B$ are strictly positive, then
\begin{equation}\label{e2}
S_q^f(A|B)=B S_{q-1}^f\left(B^{-1}|A^{-1}\right) B\,.
\end{equation}
In particular, $S_0^f(A|B)=A S_0^f\left(B^{-1}|A^{-1}\right) B$.
\end{lemma}
\begin{proof} Since $Xg(X^*X)=g(XX^*)X$ for every $X\in\mathbb{B}(\mathscr{H})$ and every continuous function $g$ on $[0, \|X\|^2]$ \cite[Lemma 1.7]{FMPS}, considering $X=B^{1/2}A^{-1/2}$ we have
\begin{align*}
S_q^f(A|B)&=A^{\frac{1}{2}}\left(A^{-\frac{1}{2}}BA^{-\frac{1}{2}}\right)^q f\left(A^{-\frac{1}{2}}BA^{-\frac{1}{2}}\right)A^{\frac{1}{2}}\\
&=A^{\frac{1}{2}}\left(X^*X\right)^q f\left(X^*X\right)A^{\frac{1}{2}}\\
&=B^{\frac{1}{2}}B^{-\frac{1}{2}}A^{\frac{1}{2}} \left(X^*X\right)^q f\left(X^*X\right)A^{\frac{1}{2}}B^{-\frac{1}{2}}B^{\frac{1}{2}}\\
&=B^{\frac{1}{2}}{X^*}^{-1}\left(X^*X\right)^q f\left(X^*X\right)X^{-1}B^{\frac{1}{2}}\\
&=B^{\frac{1}{2}}X(X^*X)^{-1}\left(X^*X\right)^q f\left(X^*X\right)(X^*X)^{-1}X^*B^{\frac{1}{2}}\\
&=B^{\frac{1}{2}}X(X^*X)^{q-2} f\left(X^*X\right)X^*B^{\frac{1}{2}}\\
&=B^{\frac{1}{2}}(XX^*)^{q-2} X f\left(X^*X\right)X^*B^{\frac{1}{2}}\\
&=B^{\frac{1}{2}}(XX^*)^{q-2} f\left(XX^*\right) XX^*B^{\frac{1}{2}}\\
&=B^{\frac{1}{2}}(XX^*)^{q-1} f\left(XX^*\right)B^{\frac{1}{2}}\\
&=B^{\frac{1}{2}}\left(B^{\frac{1}{2}}A^{-1}B^{\frac{1}{2}}\right)^{q-1} f\left(B^{\frac{1}{2}}A^{-1}B^{\frac{1}{2}}\right)B^{\frac{1}{2}}\\
&=B S_{q-1}^f(B^{-1}|A^{-1}) B\,,
\end{align*}
as desired.

Also,
\begin{align*}
S_0^f(A|B)&=B S_{-1}^f(B^{-1}|A^{-1}) B\\
&=B\left[B^{-\frac{1}{2}}\left(B^{\frac{1}{2}}A^{-1}B^{\frac{1}{2}}\right)^{-1}f\left(B^{\frac{1}{2}}A^{-1}B^{\frac{1}{2}}\right)B^{-\frac{1}{2}}\right]B\\
&=A \left[B^{-\frac{1}{2}} f\left(B^{\frac{1}{2}}A^{-1}B^{\frac{1}{2}}\right)B^{-\frac{1}{2}}\right]B\\
&=A S_0^f\left(B^{-1}|A^{-1}\right) B\,.
\end{align*}
\end{proof}
The above lemma says that if $B$ is also invertible, then we can define $S_q^f(A|B)$ by \eqref{e2}. Furthermore, if $A$ and $B$ commute, then
$$
S_q^f(A|B)=A \left(A^{-1}B\right)^q f\left(A^{-1}B\right)\,.
$$
Note that, in general case, the generalized relative operator entropy for noninvertible positive operators does not always exist. For instance, give $f(t)=\log t$ and $q=0$, consequently $S_0^f(I_{\mathscr H}|\epsilon I_{\mathscr H})=(\log\epsilon)I_{\mathscr H}$ is not bounded below and hence $S_0^f(I_{\mathscr H}|0)$ does not make sense. For more information, see \cite{FMPS}.\\
Now, we have the following lemmas.
\begin{lemma}
Let $\mathbf{A}=(A_t)_{t\in T}$ and $\mathbf{B}=(B_t)_{t\in T}$ be continuous fields of strictly positive operators. Then
\begin{align}\label{kian12}
\int_T(A_s \natural_p B_s)d\mu(s) \le
\left(\int_TA_sd\mu(s)\right) \natural_p
\left(\int_TB_sd\mu(s)\right),
\end{align}
where $p\in[0,1]$.
\end{lemma}
\begin{proof}
For continuous fields of strictly positive operators $\mathbf{A}=(A_t)_{t\in T}$ and $\mathbf{B}=(B_t)_{t\in T}$, we take the  positive unital linear map $\Phi(X)=\int_TC^*XCd\mu(t)\,\,(X\in{\mathscr A})$, where $C=B_t^{\frac{1}{2}}\left(\int_TB_sd\mu(s)\right)^{-\frac{1}{2}}$. Thus for $p\in[0,1]$ we have
{\footnotesize\begin{align*}
&\left(\int_TA_td\mu(t)\right) \natural_p\left(\int_TB_sd\mu(s)\right)\\
&=\left(\int_TB_sd\mu(s)\right)^{\frac{1}{2}}
\left(\left(\int_TB_sd\mu(s)\right)^{-\frac{1}{2}}\int_TA_td\mu(t)\left(\int_TB_sd\mu(s)\right)^{-\frac{1}{2}}\right)^p\left(\int_TB_sd\mu(s)\right)^{\frac{1}{2}}\\
&=\left(\int_TB_sd\mu(s)\right)^{\frac{1}{2}}
\left(\int_T\left(\int_TB_sd\mu(s)\right)^{-\frac{1}{2}}B_t^{\frac{1}{2}}(B_t^{-\frac{1}{2}}A_tB_t^{-\frac{1}{2}})B_t^{\frac{1}{2}}
\left(\int_TB_sd\mu(s)\right)^{-\frac{1}{2}}d\mu(t)
\right)^p\\
&\qquad\qquad\times\left(\int_TB_sd\mu(t)\right)^{\frac{1}{2}}\\&
=\left(\int_TB_sd\mu(s)\right)^{\frac{1}{2}}
\left(\int_TC^*B_t^{-\frac{1}{2}}A_tB_t^{-\frac{1}{2}}Cd\mu(t)
\right)^p\left(\int_TB_sd\mu(t)\right)^{\frac{1}{2}}\\&=
\left(\int_TB_sd\mu(s)\right)^{\frac{1}{2}}\left(\Phi\left(B_t^{-\frac{1}{2}}A_tB_t^{-\frac{1}{2}}\right)\right)^p\left(\int_TB_sd\mu(s)\right)^{\frac{1}{2}}
\\&\geq\left(\int_TB_sd\mu(s)\right)^{\frac{1}{2}}\Phi\left(\left(B_t^{-\frac{1}{2}}A_tB_t^{-\frac{1}{2}}\right)^p\right)\left(\int_TB_sd\mu(s)\right)^{\frac{1}{2}}
\qquad\qquad(\textrm{by ~\eqref{jen}})\\&=\left(\int_TB_sd\mu(s)\right)^{\frac{1}{2}}
\left(\int_TC^*\left(B_t^{-\frac{1}{2}}A_tB_t^{-\frac{1}{2}}\right)^p Cd\mu(t)
\right)\left(\int_TB_sd\mu(t)\right)^{\frac{1}{2}}\\&=\int_TB_t^{\frac{1}{2}}\left(B_t^{-\frac{1}{2}}A_tB_t^{-\frac{1}{2}}\right)^p B_t^{\frac{1}{2}}d\mu(t)
\\&=\int_T(A_t \natural_p B_t)d\mu(t)\,.
\end{align*}}
\end{proof}
\begin{lemma}\label{2.3}
If $X, C_s\in\mathscr A\,\,(s\in T)$ such that $0<m \le X \le M$, $f:(0, \infty)\to[0, \infty)$ is an operator monotone function and $t_0\in[m, M]$, then
\begin{align*}
f&\left(\int_T C_s^*XC_sd\mu(s)+t_0\Big(I_{\mathscr H}-\int_TC_s^*C_sd\mu(s)\Big)\right)\\
&\ge \int_TC_s^*f(X)C_sd\mu(s)+f(t_0)\Big(I_{\mathscr H}-\int_TC_s^*C_sd\mu(s)\Big)\,,
\end{align*}
where  $\int_TC_s^*C_sd\mu(s)\leq I_{\mathscr H}$.
\end{lemma}
\begin{proof}
We put $D=(I_{\mathscr H}-\int_TC_s^*C_sd\mu(s))^\frac{1}{2}$. Assume that the positive unital linear map $\Phi\left({\rm diag}(X,Y)\right)=\int_TC_s^*XC_sd\mu(s)+D^*YD\,\,(X,Y\in\mathscr{A})$, where
$${\rm diag}(X,Y)=\left[\begin{array}{cc}
                         X &  0\\0 & Y
 \end{array}\right]\,.
$$
Using inequality \eqref{jen} and the operator monotonicity of $f$ we have
\begin{align*}
f&\left(\int_TC_s^*XC_sd\mu(s)+t_0\left(I_{\mathscr H}-\int_TC_s^*C_sd\mu(s)\right)\right)\\
&=f\left(\Phi\left(\left[\begin{array}{cc}
                         X &  0\\0 & t_0
 \end{array}\right]\right)\right)\\
 &\geq\left(\Phi\left(\left[\begin{array}{cc}
         f(X) &  0\\0 & f(t_0)
 \end{array}\right]\right)\right)\,\qquad(\textrm{by \eqref{jen}})\\
 &=\int_TC_s^*f(X)C_sd\mu(s)+D^*f(t_0)D\,,
\end{align*}
whenever $0<m\leq X \leq M$ and $t_0\in[m,M]$. Therefore we get the desired inequality.
\end{proof}
In the next theorem we have an extension of \eqref{ShEnIn1}.
\begin{theorem}\label{tShEnIn000}
Let $\mathbf{A}=(A_t)_{t\in T},\mathbf{B}=(B_t)_{t\in T}$ be continuous fields of strictly positive operators such that $0<m A_s \leq B_s \leq M A_s\,\,(s\in T)$  for some positive real numbers $m, M$, where $m<1<M$, and $\int_TA_sd\mu(s)=\int_TB_sd\mu(s)=I_{\mathscr H}$, $f: (0,\infty) \to [0,\infty)$ be operator concave and $p\in[0,1]$. Then
{\footnotesize
\begin{align}\label{ShEnIn}
f&\left[\int_T(A_s\natural_{p+1}B_s)d\mu(s)+t_0\left(I_{\mathscr H}-\int_TA_s\natural_pB_sd\mu(s)\right)\right]-f(t_0)\left(I_{\mathscr H}-\int_TA_s\natural_pB_sd\mu(s)\right)\nonumber\\
&\ge \widetilde{S}_p^f(\mathbf{A}|\mathbf{B})\,.
\end{align}}
\end{theorem}
\begin{proof}
It follows from \eqref{kian12} and $\int_TA_sd\mu(s)=\int_TB_sd\mu(s)=I_{\mathscr H}$ that $\int_TA_s \natural_{p} B_s d\mu(s)\le I_{\mathscr H}\,\,(p\in[0,1])$.
\begin{align*}
f\Bigg[&\int_T(A_s\natural_{p+1}B_s)d\mu(s)+t_0\left(I_{\mathscr H}-\int_TA_s\natural_pB_sd\mu(s)\right)\Bigg]\\
=&f\Bigg[\int_T\left(\Big(A_s^{-\frac{1}{2}}B_sA_s^{-\frac{1}{2}}\Big)^{\frac{p}{2}}A_s^{\frac{1}{2}}\right)^*
\Big(A_s^{-\frac{1}{2}}B_sA_s^{-\frac{1}{2}}\Big)\left(\Big(A_s^{-\frac{1}{2}}B_sA_s^{-\frac{1}{2}}\Big)^{\frac{p}{2}}A_s^{\frac{1}{2}}\right)d\mu(s)\\
&+t_0\left(I_{\mathscr H}-\int_TA_s\natural_p B_sd\mu(s)\right)\Bigg]\\
\ge &\int_TA_s^{\frac{1}{2}}\left(A_s^{-\frac{1}{2}}B_sA_s^{-\frac{1}{2}}\right)^{\frac{p}{2}}
f\left(A_s^{-\frac{1}{2}}B_sA_s^{-\frac{1}{2}}\right)\left(A_s^{-\frac{1}{2}}B_sA_s^{-\frac{1}{2}}\right)^{\frac{p}{2}}
A_s^{\frac{1}{2}}d\mu(s)\\
&+f(t_0)\left(I_{\mathscr H}-\int_TA_s\natural_pB_sd\mu(s)\right)\quad\quad\quad\quad\mbox{(by Lemma \ref{2.3})}\\
=&\int_TA_s^{\frac{1}{2}}\left(A_s^{-\frac{1}{2}}B_sA_s^{-\frac{1}{2}}\right)^p
f\left(A_s^{-\frac{1}{2}}B_sA_s^{-\frac{1}{2}}\right)A_s^{\frac{1}{2}}d\mu(s)+f(t_0)\left(I_{\mathscr H}-\int_TA_s\natural_pB_sd\mu(s)\right)\\
=&\int_T{S}_p^f(A_s|B_s)d\mu(s)+f(t_0)\left(I_{\mathscr H}-\int_TA_s\natural_pB_sd\mu(s)\right)\,.
\end{align*}
\end{proof}
In the next theorem, we present the lower and upper bound of the generalized relative operator entropy.
\begin{theorem}\label{t2}
With above notations, the following statements hold:
\begin{enumerate}
\item[(i)] $\widetilde{S}_q^f(\mathbf{A}|\mathbf{B}) \ge 0$.
\item[(ii)] If $f(t) \le t-1$, then $\widetilde{S}_q^f(\mathbf{A}|\mathbf{B}) \le \int_T(A_s\natural_{q+1} B_s-A_s\natural_q B_s)d\mu(s)$. In particular, $\widetilde{S}_0^f(\mathbf{A}|\mathbf{B}) \le \int_T(B_s-A_s)d\mu(s)$ and $\widetilde{S}_1^f(\mathbf{A}|\mathbf{B}) \le \int_T(B_sA_s^{-1}B_s-B_s)d\mu(s)$.
\end{enumerate}
\end{theorem}
\begin{proof}
(i) Since $f$ is a continuous nonnegative function, $X^qf(X) \ge 0$ for every $X \ge 0$ and $q\in\mathbb{R}$. Hence
$$
\left(A_s^{-\frac{1}{2}}B_sA_s^{-\frac{1}{2}}\right)^q f\left(A_s^{-\frac{1}{2}}B_sA_s^{-\frac{1}{2}}\right) \ge 0\,.
$$
Consequently, $\widetilde{S}_q^f(\mathbf{A}|\mathbf{B}) \ge 0$.\\
\\
(ii) Since $f(t) \le t-1$, we have
\begin{align*}
\widetilde{S}_q^f(\mathbf{A}|\mathbf{B})&=\int_TA_s^{\frac{1}{2}}\left(A_s^{-\frac{1}{2}}B_sA_s^{-\frac{1}{2}}\right)^q f\left(A_s^{-\frac{1}{2}}B_sA_s^{-\frac{1}{2}}\right)A_s^{\frac{1}{2}}d\mu(s)\\
& \le  \int_TA_s^{\frac{1}{2}}\left(A_s^{-\frac{1}{2}}B_sA_s^{-\frac{1}{2}}\right)^q \left(A_s^{-\frac{1}{2}}B_sA_s^{-\frac{1}{2}}-I_{\mathscr H}\right)A_s^{\frac{1}{2}}d\mu(s)\\
& = \int_T\left(A_s\natural_{q+1} B_s-A_s\natural_q B_s\right)d\mu(s)\,.
\end{align*}
Hence
$$
\widetilde{S}_0^f(\mathbf{A}|\mathbf{B}) \le \int_T\left(A_s\natural_1 B_s-A_s\natural_0 B_s\right)d\mu(s) = \int_T\left(B_s-A_s\right)d\mu(s)\,,
$$
and
\begin{align*}
\widetilde{S}_1^f(\mathbf{A}|\mathbf{B}) & \le \int_T\left(A_s\natural_2 B_s-A_s\natural_1 B_s\right)d\mu(s)\\
&=\int_T\left[A_s^{\frac{1}{2}}\left(A_s^{-\frac{1}{2}}B_sA_s^{-\frac{1}{2}}\right)^2 A_s^{\frac{1}{2}}-A_s^{\frac{1}{2}}\left(A_s^{-\frac{1}{2}}B_sA_s^{-\frac{1}{2}}\right)A_s^{\frac{1}{2}}\right]d\mu(s)\\
&=\int_T\left(B_sA_s^{-1}B_s-B_s\right)d\mu(s)\,.
\end{align*}
\end{proof}
\begin{corollary}\cite[Theorem 5.12]{FMPS}\label{co11}
Assume that $A$ and $B$ are two strictly positive operators in $\mathbb{B}(\mathscr{H})$. Then the relative operator entropy is upper bounded; i.e. $S(A|B) \le B-A$.
\end{corollary}
\begin{proof}
By taking $T=\{1\}$, $f(t)=\log t$ and $q=0$ in Theorem \ref{t2} (ii), it follows from the Klein inequality $\log t \le t-1$.
\end{proof}
\begin{corollary}[Information inequality]\cite[Lemma 3.1]{G}
	Given two probability mass functions $\{a_j\}$ and $\{b_j\}$, that is, two countable or finite sequences of nonnegative numbers that sum to one, then
	\begin{equation}\label{eq111}
    \sum_j a_j\log\frac{a_j}{b_j} \ge 0\,,
	\end{equation}
	with equality if and only if $a_j=b_j$, for all $j$.
\end{corollary}
\begin{proof}\label{r1}
	If we take $A$ and $B$ in Corollary \ref{co11} as follows
	$$
	A=\begin{pmatrix}
	a_1  &     0    &     0    & \cdots \\
	0    &   a_2   &     0    & \cdots \\
	0    &     0     &   a_3  & \cdots\\
	\vdots & \vdots  & \vdots & \ddots
	\end{pmatrix},\quad\quad
	B=\begin{pmatrix}
	b_1  &     0    &     0    & \cdots \\
	0    &   b_2   &     0    & \cdots \\
	0    &     0     &   b_3  & \cdots\\
	\vdots & \vdots  & \vdots & \ddots\\
	\end{pmatrix}\,,
	$$
	then, we have
	\begin{align*}
    S(A|B)
	& = \begin{pmatrix}
	a_1  &     0    &     0    & \cdots \\
	0    &   a_2   &     0    & \cdots \\
	0    &     0     &   a_3  & \cdots\\
	\vdots & \vdots  & \vdots & \ddots
	\end{pmatrix}
	\log
	\begin{pmatrix}
	\frac{b_1}{a_1} &     0                   &     0                  & \cdots \\
	0              & \frac{b_2}{a_2} &     0                  & \cdots\\
	0                      &      0                 & \frac{b_3}{a_3} & \cdots\\
	\vdots               & \vdots               & \vdots                & \ddots
	\end{pmatrix}\\
	& = \begin{pmatrix}
	a_1\log\left(\frac{b_1}{a_1}\right) &                                 0                                  &                              0                     & \cdots \\
	0                                                    & a_2\log\left(\frac{b_2}{a_2}\right) &     0                  & \cdots\\
	0                        &     0       & a_3\log\left(\frac{b_3}{a_3}\right) &           \cdots\\
	\vdots               & \vdots               & \vdots                & \ddots
	\end{pmatrix}\,.
	\end{align*}
	Consequently,
	\begin{align*}
	&\left\langle S(A|B) \begin{pmatrix}
	1\\
	1\\
	1\\
	\vdots
	\end{pmatrix},
	\begin{pmatrix}
	1\\
	1\\
	1\\
	\vdots
	\end{pmatrix}\right\rangle\\
	&=\left\langle
	\begin{pmatrix}
	a_1\log\left(\frac{b_1}{a_1}\right) &                                 0                                  &                              0                     & \cdots \\
	0                                                    & a_2\log\left(\frac{b_2}{a_2}\right) &     0                  & \cdots\\
	0                        &     0       & a_3\log\left(\frac{b_3}{a_3}\right) &           \cdots\\
	\vdots               & \vdots               & \vdots                & \ddots
	\end{pmatrix}
	\begin{pmatrix}
	1\\
	1\\
	1\\
	\vdots
	\end{pmatrix},
	\begin{pmatrix}
	1\\
	1\\
	1\\
	\vdots
	\end{pmatrix}\right\rangle\\
	&=\sum_{j} a_j \log\left(\frac{b_j}{a_j}\right)\,.
	\end{align*}
	On the other hand, 
	\begin{align*}
	\left\langle (B-A)
	\begin{pmatrix}
	1\\
	1\\
	1\\
	\vdots
	\end{pmatrix},
	\begin{pmatrix}
	1\\
	1\\
	1\\
	\vdots
	\end{pmatrix}\right\rangle
	&=\left\langle 
	\begin{pmatrix}
	b_1-a_1  &     0    &     0    & \cdots \\
	0    &   b_2-a_2   &     0    & \cdots \\
	0    &     0     &   b_3-a_3  & \cdots\\
	\vdots & \vdots  & \vdots & \ddots
	\end{pmatrix}
		\begin{pmatrix}
	1\\
	1\\
	1\\
	\vdots
	\end{pmatrix},
	\begin{pmatrix}
	1\\
	1\\
	1\\
	\vdots
	\end{pmatrix}\right\rangle\\
	&=\sum_{j}(b_j-a_j)=0\,.
	\end{align*}
	We therefore deduce the desired inequality \eqref{eq111}.
\end{proof}

In the next theorem, we show that the generalized relative operator entropy is subadditive.
\begin{theorem}\label{t3}
For $q=0$, the generalized relative operator entropy is subadditive,
$$
\widetilde{S}_0^f(\mathbf{A}+\mathbf{B}|\mathbf{C}+\mathbf{D}) \ge \widetilde{S}_0^f(\mathbf{A}|\mathbf{C})+\widetilde{S}_0^f(\mathbf{B}|\mathbf{D})\,.
$$
\end{theorem}
\begin{proof}
Without loss of generality, we assume that $A_s$ and $B_s\quad(s\in T)$ are invertible. Put $X_s=A_s^{1/2}(A_s+B_s)^{-1/2}$ and $Y_s=B_s^{1/2}(A_s+B_s)^{-1/2}$. Therefore,
{\footnotesize
\begin{align*}
X_s^*X_s+Y_s^*Y_s=&\left[(A_s+B_s)^{-\frac{1}{2}}A_s^{\frac{1}{2}}\right]\left[A_s^{\frac{1}{2}}(A_s+B_s)^{-\frac{1}{2}}\right]
+\left[(A_s+B_s)^{-\frac{1}{2}}B_s^{\frac{1}{2}}\right]\left[B_s^{\frac{1}{2}}(A_s+B_s)^{-\frac{1}{2}}\right]\\
=&(A_s+B_s)^{-\frac{1}{2}}A_s(A_s+B_s)^{-\frac{1}{2}}+(A_s+B_s)^{-\frac{1}{2}}B_s(A_s+B_s)^{-\frac{1}{2}}\\
=&(A_s+B_s)^{-\frac{1}{2}}(A_s+B_s)(A_s+B_s)^{-\frac{1}{2}}\\
=&I_{\mathscr H}\,,
\end{align*}}
this implies that
{\footnotesize
\begin{eqnarray*}
&&\widetilde{S}_0^f(\mathbf{A}+\mathbf{B}|\mathbf{C}+\mathbf{D})=\int_T{S}_0^f(A_s+B_s|C_s+D_s)d\mu(s)\\
&&=\int_T(A_s+B_s)^{\frac{1}{2}} f\left[(A_s+B_s)^{-\frac{1}{2}}(C_s+D_s)(A_s+B_s)^{-\frac{1}{2}}\right](A_s+B_s)^{\frac{1}{2}}d\mu(s)\\
&&=\int_T(A_s+B_s)^{\frac{1}{2}} f\left[(A_s+B_s)^{-\frac{1}{2}}C_s(A_s+B_s)^{-\frac{1}{2}}+(A_s+B_s)^{-\frac{1}{2}}D_s(A_s+B_s)^{-\frac{1}{2}}\right]\\
&&\quad\times(A_s+B_s)^{\frac{1}{2}}d\mu(s)\\
&&=\int_T(A_s+B_s)^{\frac{1}{2}} f\left[X_s^*\left(A_s^{-\frac{1}{2}}C_sA_s^{-\frac{1}{2}}\right)X_s+Y_s^*\left(B_s^{-\frac{1}{2}}D_sB_s^{-\frac{1}{2}}\right)Y_s\right](A_s+B_s)^{\frac{1}{2}}d\mu(s)\\
&& \ge \int_T(A_s+B_s)^{\frac{1}{2}} \left[X_s^*f\left(A_s^{-\frac{1}{2}}C_sA_s^{-\frac{1}{2}}\right)X_s+Y_s^*f\left(B_s^{-\frac{1}{2}}D_sB_s^{-\frac{1}{2}}\right)Y_s\right](A_s+B_s)^{\frac{1}{2}}d\mu(s)\\
&&\qquad \hspace{8cm}(\text{by concavity of }f)\\
&&=\int_T\left[A_s^{\frac{1}{2}}f\left(A_s^{-\frac{1}{2}}C_sA_s^{-\frac{1}{2}}\right)
A_s^{\frac{1}{2}}+B_s^{\frac{1}{2}}f\left(B_s^{-\frac{1}{2}}D_sB_s^{-\frac{1}{2}}\right)B_s^{\frac{1}{2}}\right]d\mu(s)\\
&&= \int_T{S}_0^f(A_s|C_s)d\mu(s)+\int_TS_0^f(B_s|D_s)d\mu(s)\\
&&= \widetilde{S}_0^f(\mathbf{A}|\mathbf{C})+\widehat{S}_0^f(\mathbf{B}|\mathbf{D})\,.
\end{eqnarray*}}
\end{proof}
\begin{lemma}\label{2.9}
The generalized relative operator entropy is homogenous; i.e. for any $0\neq \alpha \in \mathbb R$
$$
\widetilde{S}_q^f(\alpha\mathbf A|\alpha\mathbf B)=\alpha \widetilde{S}_q^f(\mathbf A|\mathbf B)\,,
$$
where $\alpha\mathbf A=(\alpha A_s)_{s\in T}$.
\end{lemma}
\begin{theorem}\label{t4}
For $q=0$, the generalized relative operator entropy is jointly concave; i.e. if $\mathbf{A}=\alpha \mathbf{A}_1+\beta \mathbf{A}_2$ and $\mathbf{B}=\alpha \mathbf{B}_1+\beta \mathbf{B}_2$ for $\alpha, \beta >0$ with $\alpha+\beta=1$, then
$$
\widetilde{S}_0^f(\mathbf{A}|\mathbf{B}) \ge \alpha \widetilde{S}_0^f(\mathbf{A}_1|\mathbf{B}_1)+\beta S_0^f(\mathbf{A}_2|\mathbf{B}_2)\,.
$$
\end{theorem}
\begin{proof}
Assume that $\mathbf{A}_k=(A_{ks})$ and $\mathbf{B}_k=(B_{ks})~~(k=1, 2,~s\in T)$. By using of subadditivity and homogeneity of the generalized relative operator entropy, we get
\begin{align*}
\widetilde{S}_0^f(\mathbf{A}|\mathbf{B})&=\widetilde{S}_0^f(\alpha \mathbf{A}_1+\beta \mathbf{A}_2|\alpha \mathbf{B}_1+\beta \mathbf{B}_2)\\
&=\int_T S_0^f(\alpha A_{1s}+\beta A_{2s}|\alpha B_{1s}+\beta B_{2s})d\mu(s)\\
& \ge \int_T\left[S_0^f(\alpha A_{1s}|\alpha B_{1s})+S_0^f(\beta A_{2s}|\beta B_{2s})\right]d\mu(s)\\
& = \int_T\left[\alpha S_0^f(A_{1s}|B_{1s})+\beta S_0^f(A_{2s}|B_{2s})\right]d\mu(s)\\
&= \alpha \widetilde{S}_0^f(\mathbf{A}_1|\mathbf{B}_1)+\beta \widetilde{S}_0^f(\mathbf{A}_2|\mathbf{B}_2)\,.
\end{align*}
\end{proof}

We say that $\mathbf{A}=(A_s)_{s\in T}$ is \textit{invertible}, if $A_s$'s are invertible for every $s\in T$.

In the following theorem, we show that the generalized relative operator entropy has informational monotonicity.
\begin{theorem}\label{t5}
For $q=0$ and $\Phi \in \mathbf{P}_N[\mathbb{B}(\mathscr{H}),\mathbb{B}(\mathscr{K})]$,
$$
\Phi\left(\widetilde{S}_0^f(\mathbf{A}|\mathbf{B})\right) \le \widetilde{S}_0^f\left(\Phi(\mathbf{A})|\Phi(\mathbf{B})\right)\,.
$$
\end{theorem}
\begin{proof}
Assume that $\mathbf{A}$ is invertible. Then so does $\Phi(\mathbf{A})=\big(\Phi(A_s)\big)_{s\in T}$. Define
$$
\Psi(X)=\Phi(A_s)^{-\frac{1}{2}}\Phi\left(A_s^{\frac{1}{2}}XA_s^{\frac{1}{2}}\right)\Phi(A_s)^{-\frac{1}{2}}\,.
$$
So $\Psi$ is a normalized positive linear map. Consequently
{\footnotesize
\begin{align*}
\Phi\left(\widetilde{S}_0^f (\mathbf{A}|\mathbf{B})\right)&=\int_T\Phi\left(S_0^f(A_s|B_s)\right)d\mu(s)\quad\quad\quad\quad\quad\quad\quad\quad\quad(\text{by \eqref{e1}})\\
&=\int_T\Phi\left(A_s^{\frac{1}{2}}f\left(A_s^{-\frac{1}{2}}B_sA_s^{-\frac{1}{2}}\right)A_s^{\frac{1}{2}}\right)d\mu(s)\\
&=\int_T\Phi(A_s)^{\frac{1}{2}}\left[\Phi(A_s)^{-\frac{1}{2}} \Phi\left(A_s^{\frac{1}{2}}f\left(A_s^{-\frac{1}{2}}B_sA_s^{-\frac{1}{2}}\right)A_s^{\frac{1}{2}}\right) \Phi(A_s)^{-\frac{1}{2}}\right]\Phi(A_s)^{\frac{1}{2}}d\mu(s)\\
&=\int_T\Phi(A_s)^{\frac{1}{2}}\Psi\left(f\left(A_s^{-\frac{1}{2}}B_sA_s^{-\frac{1}{2}}\right)\right)\Phi(A_s)^{\frac{1}{2}}d\mu(s)\\
&\le\int_T\Phi(A_s)^{\frac{1}{2}}f\left(\Psi\left(A_s^{-\frac{1}{2}}B_sA_s^{-\frac{1}{2}}\right)\right)\Phi(A_s)^{\frac{1}{2}}d\mu(s)\quad\quad\quad\quad(\text{by } \eqref{jen})\\
&=\int_T\Phi(A_s)^{\frac{1}{2}}f\left[\Phi(A_s)^{-\frac{1}{2}}\Phi\left(A_s^{\frac{1}{2}}A_s^{-\frac{1}{2}}B_sA_s^{-\frac{1}{2}}A_s^{\frac{1}{2}}\right)\Phi(A_s)^{-\frac{1}{2}}\right]
\Phi(A_s)^{\frac{1}{2}}d\mu(s)\\
&=\int_T\Phi(A_s)^{\frac{1}{2}}f\left[\Phi(A_s)^{-\frac{1}{2}}\Phi(B_s)\Phi(A_s)^{-\frac{1}{2}}\right]\Phi(A_s)^{\frac{1}{2}}d\mu(s)\\
&=\int_TS_0^f\left(\Phi(A_s)|\Phi(B_s)\right)d\mu(s)\\
&= \widetilde{S}_0^f\left(\Phi(\mathbf{A})|\Phi(\mathbf{B})\right)\,.
\end{align*}}
\end{proof}
\section{Some  operator entropy inequalities}
There is an impressive method for finding inverses of some operator inequalities. It was introduced for investigation of converses of the Jensen inequality associated with convex functions, see \cite{FMPS} and \cite{MT} and references therein. We need  essence of the Mond--Pe\v{c}ari\'c method where appeared in \cite[Chapter 2]{FMPS} in some general forms.\\
If $f$ is a strictly  concave differentiable function on an interval $[m,M]$ with $m<M$ and  $\Phi:\mathbb{B}(\mathscr{H})\longrightarrow\mathbb{B}(\mathscr{K})$ is a positive unital linear map,

{\footnotesize\begin{eqnarray}\label{gama}
\mu_f=\frac{f(M)-f(m)}{M-m},\,\, \nu_f=\frac{Mf(m)-mf(M)}{M-m}\,\,{\rm~and~}\,\,\gamma_f=\displaystyle{\max\left\{\frac{f(t)}{\mu_f t+\nu_f}: m\leq t\leq M\right\}},
\end{eqnarray}}
  then
\begin{eqnarray}\label{mond1}
f(\Phi(A))\leq\gamma_f \Phi(f(A)).
\end{eqnarray}

\begin{lemma}\label{reverses1}
Assume that $X,C_s\in\mathscr{A}\,\,(s\in T)$ such that $0<m\leq X \leq M$, $f:(0,\infty) \to [0,\infty)$ is an operator monotone function, $t_0\in[m,M]$ and $\gamma_f$ is given by \eqref{gama}.
 Then
\begin{align*}
f & \left(\int_TC_s^*XC_sd\mu(s)+t_0\left(I_{\mathscr H}-\int_TC_s^*C_sd\mu(s)\right)\right)\\
& \leq \gamma_f\left[\int_TC_s^*f(X)C_sd\mu(s)+f(t_0)\left(I_{\mathscr H}-\int_TC_s^*C_sd\mu(s)\right)\right]\,,
\end{align*}
where  $\int_TC_s^*C_sd\mu(s)\leq I_{\mathscr H}$.
\end{lemma}
\begin{proof}
By using \eqref{mond1}, the proof is similar to Lemma \ref{2.3}.
\end{proof}

\begin{theorem}\label{tShEnIn}
Let $\mathbf{A}=(A_t)_{t\in T}, \mathbf{B}=(B_t)_{t\in T}$ be continuous fields of strictly positive operators such that  $0<m A_s \leq B_s \leq M A_s\,\,(s\in T)$  for some positive real numbers $m, M$, where $m<1<M$,  $\int_TA_sd\mu(s)=\int_TB_sd\mu(s)=I_{\mathscr H}$, $f: (0,\infty) \to [0,\infty)$ be  operator
concave and $p\in[0,1]$. Then
{\footnotesize
\begin{align}\label{ShEnInre}
f&\left[\int_T(A_s\natural_{p+1}B_s)d\mu(s)+t_0\left(I_{\mathscr H}-\int_TA_s\natural_pB_sd\mu(s)\right)\right]-\gamma_ff(t_0)\left(I_{\mathscr H}-\int_TA_s\natural_pB_sd\mu(s)\right)\nonumber\\
&\le \gamma_f\widetilde{S}_p^f(\mathbf{A}|\mathbf{B}),
\end{align}}
where $t_0\in[m,M]$ and $\gamma_f$ is given by \eqref{gama}.
\end{theorem}
\begin{proof}
It follows from \eqref{kian12} and $\int_TA_sd\mu(s)=\int_TB_sd\mu(s)=I_{\mathscr H}$ that $\int_TA_s \natural_{p} B_s d\mu(s)\le I_{\mathscr H}\,\,(p\in[0,1])$.
\begin{align*}
f\Bigg[&\int_T(A_s\natural_{p+1}B_s)d\mu(s)+t_0\left(I_{\mathscr H}-\int_TA_s\natural_pB_sd\mu(s)\right)\Bigg]\\
=&f\Bigg[\int_T\left(\Big(A_s^{-\frac{1}{2}}B_sA_s^{-\frac{1}{2}}\Big)^{\frac{p}{2}}A_s^{\frac{1}{2}}\right)^*
\Big(A_s^{-\frac{1}{2}}B_sA_s^{-\frac{1}{2}}\Big)\left(\Big(A_s^{-\frac{1}{2}}B_sA_s^{-\frac{1}{2}}\Big)^{\frac{p}{2}}A_s^{\frac{1}{2}}\right)d\mu(s)\\
&+t_0\left(I_{\mathscr H}-\int_TA_s\natural_p B_sd\mu(s)\right)\Bigg]\\
\le &\gamma_f\Bigg[\int_TA_s^{\frac{1}{2}}\left(A_s^{-\frac{1}{2}}B_sA_s^{-\frac{1}{2}}\right)^{\frac{p}{2}}
f\left(A_s^{-\frac{1}{2}}B_sA_s^{-\frac{1}{2}}\right)\left(A_s^{-\frac{1}{2}}B_sA_s^{-\frac{1}{2}}\right)^{\frac{p}{2}}
A_s^{\frac{1}{2}}d\mu(s)\\
&+f(t_0)\left(I_{\mathscr H}-\int_TA_s\natural_pB_sd\mu(s)\right)\Bigg]\quad\quad\quad\quad\mbox{(by Lemma \ref{reverses1})}\\
=&\gamma_f\Bigg[\int_TA_s^{\frac{1}{2}}\left(A_s^{-\frac{1}{2}}B_sA_s^{-\frac{1}{2}}\right)^p
f\left(A_s^{-\frac{1}{2}}B_sA_s^{-\frac{1}{2}}\right)A_s^{\frac{1}{2}}d\mu(s)\\
&+f(t_0)\left(I_{\mathscr H}-\int_TA_s\natural_pB_sd\mu(s)\right)\Bigg]\\
=&\gamma_f\left[\int_TS_p^f(A_s|B_s)d\mu(s)+f(t_0)\left(I_{\mathscr H}-\int_TA_s\natural_pB_sd\mu(s)\right)\right]\,.
\end{align*}
\end{proof}
For the discrete case $T=\{1,2,\cdots,n\}$, we give a reverse of \eqref{ShEnIn}.
\begin{corollary}\label{tShEnIn}
Let  $0<m A_j \leq B_j \leq M A_j\,\,(1\leq j\leq n)$  for some positive real numbers $m, M$  such that $m<1<M$,  $\sum_{j=1}^nA_j=\sum_{j=1}^nB_j=I_{\mathscr H}$ and $f: (0,\infty) \to [0,\infty)$ be  operator
concave. Then
\begin{align}\label{ShEnInre}
f&\left[\sum_{j=1}^n(A_j\natural_{p+1}B_j)+t_0\left(I_{\mathscr H}-\sum_{j=1}^nA_j\natural_pB_j\right)\right]-\gamma_ff(t_0)\left(I_{\mathscr H}-\sum_{j=1}^nA_j\natural_pB_j\right)\nonumber\\
&\le \gamma_fS_p^f(\mathbf{A}|\mathbf{B})\,,
\end{align}
where $t_0\in[m,M]$, $p\in[0,1]$ and $\gamma_f$ is given by \eqref{gama}.
\end{corollary}
There is another method to find a reverse of the Choi-Davis-Jensen inequality.
If $f$ is a strictly concave differentiable function on an interval $[m,M]$ with $m<M$ and  $\Phi$ is a unital positive linear map, then
\begin{eqnarray}\label{mond1o1}
\zeta_f I_{\mathscr H}+ \Phi(f(A)) \geq f(\Phi(A)),
\end{eqnarray}
where
{\footnotesize\begin{eqnarray}\label{beta}
\mu_f=\frac{f(M)-f(m)}{M-m},\,\, \nu_f=\frac{Mf(m)-mf(M)}{M-m}\,\,{\rm~and~}\,\,\zeta_f=\max_{m\leq t\leq M}\left\{f(t)-\mu_ft-\nu_f\right\},
\end{eqnarray}}
$A \in {\mathbb B}({\mathscr H})$ is a self-adjoint operator with spectrum in $[m,M]$; see \cite[p. 101]{FMPS}.
\begin{lemma}\label{reverses2}
 Let $X,C_s\in\mathscr{A}\,\,(s\in T)$ such that $\int_TC_s^*C_sd\mu(s)\le I_{\mathscr H}$ and $0<m\leq X \leq M$,  let $f:(0,\infty) \to [0,\infty)$ be an operator monotone function and $t_0\in[m,M]$.
 Then
  \begin{align*}
&f\left(\int_TC_s^*XC_sd\mu(s)+t_0\left(I_{\mathscr H}-\int_TC_s^*C_sd\mu(s)\right)\right)\\
&\leq\int_TC_s^*f(X)C_sd\mu(s)+f(t_0)\left(I_{\mathscr H}-\int_TC_s^*C_sd\mu(s)\right)+\zeta_f,
  \end{align*}
     where $\zeta_f$ is given by \eqref{beta}.
\end{lemma}
 \begin{proof}
We put $D=(I_{\mathscr H}-\int_TC_s^*C_sd\mu(s))^\frac{1}{2}$.
 Applying  inequality \eqref{mond1o1} for the positive unital linear map $\Phi\left({\rm diag}(X,Y)\right)=\int_TC_s^*XC_sd\mu(s)+D^*YD\,\,(X,Y\in\mathscr{A})$ we get desired inequality.
\end{proof}
Using Lemma \ref{reverses2}, by the same argument in the proof of Theorem \ref{tShEnIn} we give the next result.
\begin{theorem}\label{tShEnIn34545}
Let $\mathbf{A}=(A_t)_{t\in T},\mathbf{B}=(B_t)_{t\in T}$ be continuous fields of strictly positive operators such that  $0<m A_s \leq B_s \leq M A_s\,\,(s\in T)$  for some positive real numbers $m, M$, where $m<1<M$,  $\int_TA_sd\mu(s)=\int_TB_sd\mu(s)=I_{\mathscr H}$, $f: (0,\infty) \to [0,\infty)$ be  operator
concave, $t_0\in[m,M]$ and $p\in[0,1]$. Then
{\footnotesize
\begin{align}\label{ShEnInre}
f&\left[\int_T(A_s\natural_{p+1}B_s)d\mu(s)+t_0\left(I_{\mathscr H}-\int_TA_s\natural_pB_sd\mu(s)\right)\right]-f(t_0)\left(I_{\mathscr H}-\int_TA_s\natural_pB_sd\mu(s)\right)\nonumber\\
&\le \widetilde{S}_p^f(\mathbf{A}|\mathbf{B})+\zeta_f,
\end{align}}
where $\zeta_f$ is given by \eqref{beta}.
\end{theorem}
In the discrete case, we get a reverse of  inequality \eqref{ShEnIn}.
\begin{corollary}\label{tShEnIn345}
Let  $0<m A_j \leq B_j \leq M A_j\,\,(1\leq j\leq n)$  for some positive real numbers $m, M$  such that $m<1<M$,  $\sum_{j=1}^nA_j=\sum_{j=1}^nB_j=I_{\mathscr H}$ and $f: (0,\infty) \to [0,\infty)$ be  operator concave, $p\in[0,1]$ and $t_0\in[m,M]$. Then
\begin{align}\label{ShEnInre}
f&\left[\sum_{j=1}^n(A_j\natural_{p+1}B_j)+t_0\left(I_{\mathscr H}-\sum_{j=1}^nA_j\natural_pB_j\right)\right]-f(t_0)\left(I_{\mathscr H}-\sum_{j=1}^nA_j\natural_pB_j\right)\nonumber\\
&\le S_p^f(\mathbf{A}|\mathbf{B})+\zeta_f,
\end{align}
where $\zeta_f$ is given by \eqref{beta}.
\end{corollary}
Using Theorem \ref{tShEnIn345} for the operator monotone functions $f(t)=-t\log t$ and $g(t)=\log t$, respectively, we get the following example.
\begin{example}\label{coro}
Let  $0<m A_j \leq B_j \leq M A_j\,\,(1\leq j\leq n)$  for some positive real numbers $m, M$  such that $m<1<M$, $\sum_{j=1}^nA_j=\sum_{j=1}^nB_j=I_{\mathscr H}$ and $t_0\in[m,M]$. Then
{\footnotesize\begin{align*}
&\left[\sum_{j=1}^n(A_j\natural_{p+1}B_j)
+t_0\left(I_{\mathscr H}-\sum_{j=1}^nA_j\natural_pB_j\right)\right]\log\left[\sum_{j=1}^n(A_j\natural_{p+1}B_j)
+t_0\left(I_{\mathscr H}-\sum_{j=1}^nA_j\natural_pB_j\right)\right]\\&\qquad-t_0\log(t_0)\left(I_{\mathscr H}-\sum_{j=1}^nA_j\natural_pB_j\right)\\
&\ge S_1(\mathbf{A}|\mathbf{B})+L(1/m,1/M)^{-1}-I(m,M)
\end{align*}}
and
{\footnotesize\begin{align*}
\log&\left[\sum_{j=1}^n(A_j\natural_{p+1}B_j)+t_0\left(I_{\mathscr H}-\sum_{j=1}^nA_j\natural_pB_j\right)\right]-\log(t_0)\left(I_{\mathscr H}-\sum_{j=1}^nA_j\natural_pB_j\right)\\
&\le S(\mathbf{A}|\mathbf{B})+\log\left[\frac{1}{e}\left(\frac{M^m}{m^M}\right)^{\frac{1}{M-m}}L(m,M)\right],
\end{align*}}
where $L(a,b)=\begin{cases}\frac{b-a}{\log b-\log a} & ;a\neq b\\ a & ; a=b\end{cases}$ and $I(a,b)=\begin{cases}\frac{1}{e}\left(\frac{b^b}{a^a}\right)^{\frac{1}{b-a}} & ;a \neq b\\ a & ;a=b\end{cases}$  are the Logarithmic mean and the Identric mean of positive real numbers $a$ and $b$, respectively.
\end{example}
\bibliographystyle{amsplain}

\end{document}